\documentclass{amsart}

\usepackage{amssymb,amsmath, amscd, mathrsfs, yfonts, bbm, graphics, dsfont, footmisc }

\title[] {Asymptotic free independence and entry permutations for Gaussian random matrices}

\author{Mihai Popa }

\address{Department of Mathematics, University of Texas at San Antonio, One UTSA Circle
San Antonio, Texas 78249, USA, and}
\address{``Simon Stoilow'' Institute of Mathematics of the Romanian Academy, P.O. Box 1-764, 014700 Bucharest, Romania}
\email{mihai.popa@utsa.edu}

\newtheorem{claim}{}[section]
\newtheorem{defn}[claim]{Definition}
\newtheorem{thm}[claim]{Theorem}
\newtheorem{lemma}[claim]{Lemma}
\newtheorem{remark}[claim]{Remark}

\newtheorem{cor}[claim]{Corollary}

\newcommand{\tr}{\textrm{tr}}

\newcommand{\cI}{\mathcal{I}(N, m)}

\newcommand{\oi}{\overrightarrow{i}}

\input xy
\xyoption{all}

\usepackage{graphicx}

\begin{document}

\begin{abstract}
 The paper presents conditions on entry permutations that induce asymptotic freeness when acting on Gaussian random matrices. The class of permutations described includes the matrix transpose, as well as entry permutations relevant in Quantum Information Theory and Quantum Physics.

\end{abstract}

\maketitle


\section{Introduction}

 The transpose is probably the most known matrix transformation given by a permutation of entries. The connection between the transpose and asymptotic free independence was established in \cite{mingo-popa-transpose}, when we showed the (at that time surprising) result that unitarily invariant ensembles of random matrices are asymptotically free from their transposes.
 
 In the recent years, other matrix transformations given by entry permutations become relevant in the literature. Such examples are the partial transposes, in Quantum Information Theory (see \cite{horod}, \cite{aubrun}, \cite{banica-nechita}), and the `mixing map' from Quantum Physics (see \cite{zycz}, \cite{billiard}). The present paper gives a unitary treatment of all these transforms when acting on Gaussian random matrices. Using combinatorial techniques, it is shown that they are part of a much larger and easy to describe category of entry permutations that induce asymptotic free independence. The techniques developed here were further utilized in \cite{popa-hao} and \cite{popa-hao-boolean},  in the study of random matrices with entries in non-commutative algebras.
 
 The paper is organized as follows.
  Second section presents the framework and some combinatorial lemmata that are utilized later for the main results. To be noted, the definition of asymptotic freeness used here, Definition \ref{defn:21}, requires the existence of limit distributions, same as in \cite{mingo-popa-transpose}; though, it is possible to build a similar theory without this requirement.
  Third section presents the main result of the paper, Theorem \ref{thm:1}, together with some of its applications. Theorem \ref{thm:1} gives two conditions sufficient for a family of entry permutations to induce asymptotic free independence when action on Gaussian random matrices. These conditions are shown to be satisfied by relevant classes of entry permutations. In particular, Corollary \ref{cor:trio} shows that a 
  $ N^2 \times N^2 $
  Gaussian random matrix, its 
  $ N \times N $ 
  partial transpose and the matrix obtained via applying the `mixing map' mentioned above form an asymptotically free family.
  Forth section contains two remarks on the statement of the main result. Namely Remark \ref{remark:4:1} gives an example of two families of entry permutations satisfying the second condition of the main theorem, but not the first one, and not inducing asymptotic freeness; Remark \ref{rem:4.2} gives an example showing that the conditions from Theorem \ref{thm:1} are not necessary, i.e. the class of entry permutations inducing asymptotic freeness when acting on Gaussian random matrices is actually larger than the one described in Theorem \ref{thm:1}.

\section{Preliminaries}

  Throughout this paper, by a random matrix we will understand, as in \cite{nica-speicher}, a matrix with entries in the $ \ast$-algebra
  \[
   L^{\infty - } ( \Omega, P )
    = \bigcap_{1\leq p < \infty } L^p (\Omega, P)
   \] 
where $(\Omega, P)$ is a probability space; the expectation   will be denoted by $ E $ 
(i.e. $E(\cdot) = \int \cdot dP $ ) and the $ \ast$-operation is the complex conjugate. 

If $ N $ is a positive integer, we shall denote by $ [ N ] $ the ordered set 
 $ \{ 1, 2, \dots, N \} $, 
 respectively by $ [ \pm N ] $ the set 
 $\{ 1, -1, 2, -2, \dots, N, - N \} $.
 The group of permutations on the set of couples 
 $  [ N ] \times [ N ] =\{ (i, j):\  i, j \in [N] \}$
 shall be denoted by 
  $ \mathcal{S} ( [ N]^2) $.

   For $ A $  a $ N \times N $ random matrix and 
   $ \sigma $ a permutation from
    $ \mathcal{S} ( [ N]^2) $, 
    we shall denote by 
    $ A^\sigma $
     the $ N \times N $ random matrix with each $(i, j)$-entry equal to the $ \sigma(i, j)$-entry of $A$. I.e., if $ [ A]_{i, j}$ denotes the $(i, j)$-entry of $A$, then
     \[ [A^\sigma]_{i, j} = [A]_{\sigma(i, j)}. \]
  In particular, if
         $ t \in \mathcal{S}( [ N ]^2) $
          is the permutation given by 
          $ t(i, j) = (j, i) $, 
          then $ A^t $ is the usual matrix transpose of $ A $.    
          
  Within the paper, by a noncommutative probability space it will be understood a pair
   $ (\mathfrak{A}, \varphi) $        
    where $ \mathfrak{A} $ is a unital $ \ast$-algebra over $ \mathbb{C} $ 
 and $ \varphi : \mathfrak{A} \rightarrow \mathbb{C} $ is a positive conditional expectation. No explicit assumptions will be made on a $C^\ast$- or von Neumann structure on $ \mathfrak{A} $, nor on the traciality of $ \varphi $. 
 \begin{defn}\label{defn:21}
 A family
  $\mathcal{A} = \big\{A_{k, N}: 1 \leq k \leq M, N \in \mathbb{N}  \big\} $ of random matrices such that each $ A_{k, N} $ is of size $ N \times N $ is said to be 
  \emph{asymptotically free} if there exists a non-commutative probability space
  $( \mathfrak{A}, \varphi)$ 
  and a free family $ \{ a_1, a_2, \dots, a_M \} $  from $ \mathfrak{A} $  
  such that, for any positive integer $ R $ and any $ R $-tuple 
  $ (i_1, \dots, i_R)\in [M]^R $ we have that
   \[ 
    \lim_{N \rightarrow \infty} E \circ \tr 
    \big( A_{i_1, N} A_{i_2, N} \cdots A_{i_R, N}\big) 
    = \varphi( a_{i_1} a_{i_2} \cdots a_{i_R}).
    \]
 \end{defn}  
   In particular, the definition above requires the existence of the joint limit distribution of the family $ \mathcal{A} $.
 \begin{defn}
 A  random matrix  $ G_N = [ g_{i j}]_{ 1 \leq i, j \leq N } $ is said to be Gaussian if 
 \begin{itemize}
 \item[(i)] $ g_{i, j} = \overline{ g_{j, i}}$ for all $ i, j \in [ N ] $;
 \item[(ii)]  $ \{ g_{i, j}: 1 \leq i \leq j \leq N \} $ is a family of independent, identically distributed complex $($if $i \neq j)$ or real $($if $ i = j )$ Gaussian random variables of mean 0 and variance 
 $ \displaystyle \frac{1}{\sqrt{N}} $.
 \end{itemize} 
 \end{defn}   
 
  Suppose that,  for each positive integer $ N $ and each $ k = 1, 2, \dots, m $, 
  $ \sigma_{ k, N } $ is a permutation from $ \mathcal{S} ([ N ]^2) $. The next results will discuss the expression
  \[ E \circ \tr 
  \Big( 
  G_N^{ \sigma_{1, N}} \cdot G_N^{ \sigma_{2, N}} 
  \cdots 
  G_N^{ \sigma_{m, N}} 
  \Big).
  \]
 For that we will use the connection, via the free and classical Wick's formulae (see \cite{effros-popa}, \cite{jason}), between the Gaussian, respectively semicircular random variables and pair partitions on an ordered set. 
  By a pairing or a pair-partition on $ [ m ] $ we shall understand an involution on $ [ m ] $ with no fixed points. The set of all pairings on $ [ m ] $ will be denoted by $ P_2(m)$. If $ \pi \in P_2(m)$, a block of $ \pi $ is a couple of the form $ \{ x, \pi(x)\} $ for some $ x \in [ m]$. If $ k, l \in [ m ] $ are in the same block of $ \pi $, we will say that $ \pi $ connects $ k $ and $ l $.
 For $ \pi $  a pairing from $ P_2(m) $, 
 a 4-tuple $(a, b, c, d) $ is said to be a \emph{crossing} if
 \[ a < b < c = \pi(a) < d = \pi(b). \]
  A pairing is said to be \emph{noncrossing} if it has no crossings. The set of all non-crossing pairings  $ [ m ] $ will be denoted by $ NC_2(m) $.
   
     To simplify the notations, the index $ N $ will be omitted when there is no confusion. Also, with the convention
      $ i_{  m+1 } = i_1 $,
       denote by $ \cI $ the set 
 \[ 
  \{ \oi = (i_1, i_{-1}, i_2, i_{-2}, \dots, i_{m}, i_{-m}) \in    [N]^{2m}  :\ 
        \ i_{-k}= i_{k +1} 
          \ \textrm{for}\ k \in [ m ] \}.
 \]
    Developing the trace and using Wick's formula (see \cite{jason}), we obtain
     \begin{align*}
     E \circ \tr \big( G^{\sigma_1} \cdot G^{\sigma_2} \cdots G^{\sigma_m}  \big)
     &=
     \sum_{ \oi \in \cI} \frac{1}{N}
     E \Big(    
     [ G^{\sigma_1}]_{i_1 i_{-1}}
      [ G^{\sigma_2}]_{i_2 i_{-2}}
     \cdots [ G^{ \sigma_m}]_{i_m i_{-m}}
     \Big)\\
     = & \sum_{ \sigma \in P_2 (m) } 
     \frac{1}{N}
     \sum_{ \oi \in \cI } 
     \prod_{(k, l) \in \pi }
     E \big(  g_{ \sigma_k (i_k i_{-k})} g_{\sigma_l ( i_l i_{-l})} \big) \\
     =&
     \sum_{ \sigma \in P_2 (m) }
     \frac{1}{N}
     \sum_{ \oi \in \cI }
      \prod_{(k, l) \in \pi }
     \frac{1}{N}
      \delta_{t \circ \sigma_l(i_l i_{-l})}^{\sigma_k(i_k i_{-k})}.
     \end{align*}    
   
   Therefore, denoting 
   \begin{equation}
   \label{v:V:1}
   \left\{
   \begin{array}{ll}
 v(\pi, \overrightarrow{\sigma}, \oi)
  & = 
 \displaystyle  \prod_{ (k, l) \in \pi } 
 \frac{1}{N}
 \delta_{t \circ \sigma_l(i_l i_{ -l })}^{\sigma_k(i_k i_{ - k })}\\
 \mathcal{V}(\pi, \overrightarrow{\sigma}) 
 = &
 \displaystyle \frac{1}{N}
    \sum_{ \oi \in \cI }
    v(\pi, \overrightarrow{\sigma}, \oi)\\
    & =
 \displaystyle
  N^{ - \frac{m}{2} -1 }
      \# \{ \oi \in \cI : \  
       v(  \pi , \overrightarrow{\sigma}, \oi)  \neq 0 \}
   \end{array}
   \right.
   \end{equation}
   we have  that
   \begin{equation}\label{eq:wick}
   E \circ \tr \big( G^{\sigma_1} \cdot G^{\sigma_2} \cdots G^{\sigma_m}  \big)
   = \sum_{ \pi \in P_2(m)} 
   \mathcal{V}(\pi, \overrightarrow{\sigma}) .
   \end{equation}

 Note that, by construction,
  $ v(\cdot, \cdot, \cdot) $ 
  and
  $ \mathcal{V} ( \cdot, \cdot) $
  have a property akin to traciality.
  More precisely, if 
  $ \lambda $ 
  is a circular permutation on the set 
  $ [ m ] $,
  then 
  \begin{align*}
  \mathcal{V} & ( \pi, \overrightarrow{\sigma} ) 
  =
  \mathcal{V} 
  ( \lambda \circ \pi \circ \lambda^{ -1} ,
   \overrightarrow{ \sigma } \circ \lambda)  \\ 
   v &( \pi, \overrightarrow{\sigma}, \overrightarrow{i} )
   =
   v(
   \lambda \circ \pi \circ \lambda^{ -1},
    \overrightarrow{\sigma} \circ \lambda, \overrightarrow{i} \circ \lambda 
    ), 
  \end{align*}
  where
  $ \overrightarrow{\sigma} \circ \lambda 
  = ( \sigma_{ \lambda(1)}, \sigma_{ \lambda(2)}, \dots, \sigma_{ \lambda(m)} )$
 and
 $ \overrightarrow{i} \circ \lambda
 = ( i_{ \lambda(1)}, i_{  - \lambda(1)}, \dots, 
 i_{ \lambda(m)}, i_{  - \lambda(m)}
 ) $.

 \begin{lemma}\label{lemma:ik}
 With the notations from above, if $ \pi $ is a \emph{crossing} pairing, then
 \[ \displaystyle \lim_{ N \rightarrow \infty} 
 \mathcal{V}(\pi, \overrightarrow{\sigma})  = 0. \]
 \end{lemma}  
   \begin{proof}
    Through a circular permutation, we can suppose, without restricting the generality, that 
    $ ( 1 , b, c, d) $ is a crossing of $ \pi $. 
     
     Next we will define the sets 
  $ \{ P_k \}_{ 1 \leq k \leq \frac{m}{2} } $
  via
   $ P_1 = \{ a(1), \pi ( a (1))\} $
   and 
   $ P_{k +1} =  P_k \cup \{ a(k+1), \pi ( a (k+1))\} $
   where the sequence
    $ \{ a(k)\} _{1 \leq k \leq \frac{m}{2} } $
    is given as follows.
    
  We put $ a(1) = 1 $ (hence $ P_1 = \{ 1, c \} $), let
   $ B_k = \{ l:  \  b < k \leq c \textrm{ and } l \notin P_k \} $ and
  \[
   a(k+1) = \left\{ 
   \begin{array}{l }
   \max (B_k),  \textrm{ if } B_k \neq \emptyset; \\
   \min \{l \in [ m ]:\ l \notin P_k \}, \textrm{ otherwise}. 
   \end{array}\right.
   \]  
    
    First, remark that there exists $ r \in [ m ] $ such that $ a (r) = b $. 
    
    Indeed, by construction, the set $ [ m ] $ is the disjoint union of 
    $ \{ a(k): \  1 \leq k \leq \frac{m}{2} \} $
    and 
    $ \{ \pi ( a( k )): \  1 \leq k \leq \frac{m}{2} \} $.
    Also, we have that 
    $ b < a(k+1) \leq c $ 
    if and only if 
    $ B_k \neq \emptyset $.
    If $ b = \pi ( a ( k ) ) $ for some $ k \in [ \frac{m}{2} ] $, then
    $ a(k) = \pi(b) = d $. Since $ d > c $, it follows that $ B(k -1) = \emptyset $, which would give that
    \[  d = a(k) = \min \{ l \in [ m ]: l \notin P_{ k-1} \}
    \]
    and also
     $ b = \pi(d) \notin P_{k-1} $, which contradicts $ b > d $.
     
     Moreover, $ r $ has the property that
      $ b-1 $ and $ b+1 $ are elements of $ P_{r-1}$.

      To see that, note first that $ a(r) = b \notin (b, c] $ gives
       $ B_{r-1} = \emptyset$ so $ b + 1 \in P_{r-1}$. Also, $ B_{r-1}= \emptyset$ implies that
        $ b = a(r) = \min \{ l\in [ m]:\ l \notin P(r-1) \} $ which implies
         $ b-1 \in P_{r-1}$.
     \setlength{\unitlength}{.14cm}
     \begin{equation*}
     \begin{picture}(20,13)
    
    \put(-16.5,0){1}
    
    \put(-16,3){\circle*{2}} 
    
    
    \put(-12.5,0 ){2}
    
    \put(-12,3){\circle*{2}}
    
    
    \put(-8.5,0){3}
    
    \put(-8,3){\circle*{2}}
    
    
    \put(-4.5,0){4}
    
    \put(-3.9,3){\circle*{2}}
    
    
    \put(-0.5,0){5}
    
    \put(0,3){\circle{2}}
    
    
    \put(3.5,0){6}
    
    \put(4,3){\circle*{2}}
    
    
    \put(7.5,0){7}
    
    \put(8.1,3){\circle*{2}}
    
    
    \put(11.5,0){8}

    \put(12.1,3){\circle*{2}}
    
    
    \put(15.5,0){9}
    
    \put(16.2,3){\circle*{2}}
    
    
    \put(19,0){10}
    
    \put(20.2,3){\circle*{2}}
    
    
    \put(23.5,0){11}
    
    \put(24.2,3){\circle*{2}}
    
    
    \put(27.5,0){12}
    
    \put(28.2,3){\circle*{2}}

    
     \put(-12,3){\line(0,1){9}}
     
     \put(24, 3){\line(0,1){9}}
     
    \put(24, 12){\line(-1,0){36}}

     
     \put(-8,3){\line(0,1){5}}
     
     \put(-8,8){\line(1,0){4}}
     
     \put(-4,3){\line(0,1){5}}
     
     
     \put(4,3){\line(0,1){4}}
     
     \put(4,7){\line(1,0){12}}
     
     \put(16,3){\line(0,1){4}}
     
     
     \put(8,3){\line(0,1){3}}
     
     \put(8,6){\line(1,0){4}}
     
     \put(12,3){\line(0,1){3}}

    \linethickness{.55mm}

    \put(-16,3){\line(0,1){7}}
    
    \put(20,3){\line(0,1){7}}
    
    \put(-16,10){\line(1,0){36}}

    
    \put(0,4){\line(0,1){4}}

    \put(28, 3){\line(0,1){5}}
    
    \put(28, 8){\line(-1,0){28}}
     \end{picture}
      \end{equation*}
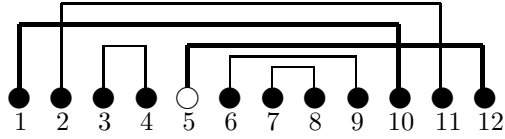
    \textbf{Figure 1}. In the diagram above, 
    with the crossing
     $(1, 5, 10, 12) $,
     we have that
    $ a(1) = 1 $,
     $ a(2) = 9 $, 
    $ a(3) = 8 $,
     $ a(4) = 2 $,
      $ a(5) = 3 $,
      $ a(6) = 5 $, 
      so 
      $ r = 6 $.

       For the next step of the proof, we need to introduce more notations. 
       First, if
      $  \oi = (i_1, i_{-1}, i_2, i_{-2}, \dots,   i_{m}, i_{-m}) $
          is an element of $ [ N]^{2m} $
       and if $ B = \{ \beta_1, \beta_2, \dots, \beta_p \} $ is a subset of $ [ m ] $ such that 
       $ \beta_k < \beta_l $ whenever $ k < l $, then 
       we denote
       \[
       \oi [  B ] = 
       ( i_{\beta_1}, i_{- \beta_1},
       i_{\beta_2}, i_{-\beta_2}, \dots,
        i_{\beta_p}, i_{-\beta_p} ) \in [ N ]^{2p}.
       \]
       
       With this notation, define
        \[
             A_{\frac{m}{2}} = \{ \oi \in I(N, m): \
              v ( \pi, \overrightarrow{\sigma}, \oi) \neq 0 \},
               \]
   and, for 
   $ 1 \leq k \leq \frac{m}{2} $,
   denote
   \[ A_k = \{ \overrightarrow{\alpha} \in [ N]^{4k} :
    \textrm{there exists } \oi \in A_{ \frac{m}{2} }
    \textrm{ such that } \overrightarrow{\alpha} = \oi [ P_k ]  
    \}
    \]
  In particular, second equation from (\ref{v:V:1}) reads
   \[ 
    \mathcal{V}(\pi, \overrightarrow{\sigma}) 
    =
    N^{ - \frac{m}{2} -1 }
      \# (A_{\frac{m}{2}}).
    \]
    So it suffice to prove that
    \begin{equation}
    \label{eq:am}
    \# (A_{\frac{m}{2}}) \leq N^{\frac{m}{2}}.
    \end{equation} 

  In order to show (\ref{eq:am}), note first that the set 
  $ I (N, m) $
   satisfies the following properties (by convention,
   $i_{\pm (m+k)}= i_{ \pm k} $):
   \begin{itemize}
   \item[$(\mathfrak{p}.1)$] For each 
   $ p \in [ m ] $ 
   and each pair 
   $ (a_1, a_2) \in [ N ]^2 $,
    there exist at most $ N $ pairs 
    $ (x_1, x_2) \in [ N ]^2$
     such that there is some 
     $ \oi \in I (N, m)  $
      with
     \[  (a_1, a_2, x_1, x_2)  =
      ( \sigma_p(i_p, i_{-p}), \sigma_{ p+1}( i_{p+1}, i_{-(p+1)})) \]
      \item[$(\mathfrak{p}.2)$] For each 
       $ p \in [ m ] $ 
       and each pair 
       $ (a_1, a_2) \in [ N ]^2 $,
        there exist at most $ N $ pairs 
        $ (x_1, x_2) \in [ N ]^2$
         such that there is some 
         $ \oi \in I (N, m)  $
          with
         \[  ( x_1, x_2, a_1, a_2)  =
          (  \sigma_{ p-1}( i_{p-1}, i_{-(p-1)}),\sigma_p(i_p, i_{-p})) \]
   \item[$(\mathfrak{p}.3)$] For each $ p \in [ m ] $
   and each $ 4 $-tuple 
   $(a_1, a_2, a_3, a_4)  \in [ N ]^4 $,
   there exists at most one pair 
   $(x_1, x_2) \in[ N]^2 $
   such that there is some 
   $ \oi \in I (N, m)  $
   with
   \[
    (a_1, a_2, x_1, x_2, a_3, a_4) 
   =
   ( \sigma_p (i_p, i_{-p}),  \sigma_{p+1} (i_{p+1}, i_{-(p+1)}), \sigma_{ p+1} ( i_{p+2}, i_{-(p+2)}) ).
   \]
   \end{itemize}
   
   First, 
       $ \oi \in I (N, m) $  
  gives that
      $ i_{-p} = i_{p+1} $ 
          and
           $ i_{-(p+1)} = i_{p+2} $,
  henceforth
  $ (x_1, x_2) = \sigma_{p+1} ( i_{ p+1}, i_{p+2}) $
        and
        $( i_{p}, i_{p+1}) =\sigma_p^{-1}(a_1, a_2)  $.
   In particular, 
   $ i_{p+1} $
    is uniquely determined by 
    $(a_1, a_2) $, 
     therefore
      $ (x_1, x_2) $ 
      is uniquely determined by the triple 
         $ (a_1, a_2, i_{ p+2}) $ 
         so the property
          $(\mathfrak{p}.1)$
   is proved.

      Similarly, 
   $( i_{p+1}, i_{p+2}) =\sigma_{p+2}^{-1}(a_3, a_4)  $, hence
    $i_{p+2} $ is uniquely determined by $(a_3, a_4) $.
     So properties $(\mathfrak{p}.2)$  and  $(\mathfrak{p}.3)$ also follow.

     Furthermore, the set 
     $ A_{ \frac{m}{2}} $
     satisfies the following property:
     \begin{itemize}
     \item[$(\mathfrak{p}.4)$] For each $ p \in [ m ] $ and each couple 
     $(a_1, a_2) \in [ N ]^2 $ 
     there exist at most one couple
      $ (x_1, x_2) $ 
      such that whenever
      $ \oi \in A_{\frac{m}{2}} $
       and
     $   \big(i_{a(p)}, i_{-a(p)}\big)  = (a_1, a_2)$,
    then
     $  \big( 
        i_{\pi(a(p))}, i_{-\pi(a(p))}
        \big) = (x_1, x_2) $. 
     \end{itemize}    
     
   This follows because if 
   $ v( \pi, \overrightarrow{\sigma}, \oi) \neq 0 $
   then
   \[
   \sigma_{\pi(a(p))} \big( 
   i_{\pi(a(p))}, i_{-\pi(a(p))}
   \big)
   = t \circ \sigma_{a(p)}
   \big(i_{a(p)}, i_{-a(p)}\big) 
   \]
   which is equivalent to
   $ ( x_1, x_2) = 
    \sigma_{\pi(a(p))}^{-1}\circ  t \circ \sigma_{a(p)}
      \big( a_1, a_2 \big)
      $.

   In particular, for $ p =1 $, property $(\mathfrak{p}.4)$ gives that 
   \begin{equation}
   \label{eq:a1}
   \# (A_1) \leq N^2.
   \end{equation}
   
   Next, remark that properties $(\mathfrak{p}.1)$, $(\mathfrak{p}.2)$ and $(\mathfrak{p}.4)$ give that
   \begin{equation}
   \label{eq:a2}
   \#(A_{k+1}) \leq N \cdot \# ( A_k ).
   \end{equation}

   To prove (\ref{eq:a2}) it suffices to show that
   given 
   $ \overrightarrow{\alpha} $ 
   an element of $ A_k $, there exist at most $ N $
   $ 4 $-tuples 
   $(x_1, x_2, x_3, x_4)  \in [ N ]^4 $
   such that there exist some
   $ \oi \in A_{ \frac{m}{2}} $
   with
   $ \oi [ P_k ] = \overrightarrow{\alpha} $
   and 
   $ \oi [ P_{k+1} \setminus P_k ]
    = (x_1, x_2, x_3, x_4) $.

    By construction,
     $ P_{k+1} \setminus P_k = \{ a( k+ 1 ), \pi(a( k+ 1 ))\} $
   with either 
   $a( k+ 1 )+1 \in P_k $ if $ B_k \neq \emptyset $,
  or $ a(k + 1 )-1 \in P_k $ if $ B_k = \emptyset $.
    So, from properties $(\mathfrak{p}.1)$ and $(\mathfrak{p}.2)$,  for each $ \overrightarrow{\alpha} \in A_k $, there are at most $ N $ couples $ (y_1, y_2) $ such that 
   $ \oi[ P_k] = \overrightarrow{\alpha} $
   and 
   $ \oi[\{ a(k)\}] = (y_1, y_2) $.
   
   Furthermore, property $(\mathfrak{p}.4)$ gives that for each couple
    $ (y_1, y_2) $ there exist at most one couple $(y_3, y_4) $ 
   such that 
   $ \oi [ \{ a(k)\}] = (y_1, y_2) $ 
   and 
   $ \oi [ \{ \pi( a(k) )\}] = (y_3, y_4) $ 
   for some $ \oi \in A_{\frac{m}{2}}.$
   Therefore the proof of (\ref{eq:a2}) is complete.   
   
   Equation (\ref{eq:a2}) give that
   \[ \# ( A_{\frac{m}{2}} ) 
      \leq N^{\frac{m}{2} - r }
      \cdot 
      \# ( A^{r}),
      \]
   while equations (\ref{eq:a1}) and (\ref{eq:a2}) imply
      \[ 
      \# (A_{ r -1}) \leq N^{r}. 
      \]
  Therefore, to complete the proof, it suffices to prove that
    \begin{equation}
    \label{eq:a3}
    \#(A_{r} )\leq \# ( A_{r-1} ).
    \end{equation}
    
   To prove (\ref{eq:a3}), remember that, as  noted above,
   $ a(r) = b $ 
   and
         $ b-1,\ b+1 \in P_{ r -1 } $. 
   So property
         $(\mathfrak{p}.3)$
     gives that
     $ \oi [ \{ a(r)\}] $ 
     is uniquely determined by 
     $ \oi [ P_{r-1}] $.
  Furthermore, property
   $(\mathfrak{p}.4)$
    gives that 
        $ \oi [ \{ \pi (a(r))\} ] $ 
        is, in turn, uniquely determined by 
        $ \oi [ \{ a(r)\}] $. 
   So
    $ \oi [ P_r ]  $
    is uniquely determined by
     $ \pi[ p_{r-1}] $,
     and the conclusion follows.
       \end{proof}   
       
   For 
   $ \sigma $
   and 
   $ \tau $
   two permutations in 
   $ \mathcal{S} ( [ N ]^2 ) $, 
   we denote
   \begin{align*}
   \mathfrak{j}( \sigma: \tau ) = \# \big\{ 
   (i, j, k) \in [ N ]^3 : \sigma ( i, j ) = t \circ \tau \circ t 
    ( k, j ) 
   \big\}.
   \end{align*}    
 With this notation, we have the following result.

  \begin{lemma}\label{lemma:ik1} 
   Let 
   $ \pi \in NC_2 (m ) $ 
   be such that 
   $ \pi ( k ) = k + 1 $
   for some 
   $ 1 \leq k \leq m -1 $.
   If
   $
    \mathfrak{j}( \sigma_{ k, N}: \sigma_{ k+1, N } ) = o ( N^2 ) ,
    $
   then
    \[
      \lim_{ N \rightarrow \infty } 
      \mathcal{V}(\pi, \overrightarrow{\sigma} )= 0.
      \]
  \end{lemma}

 \begin{proof}
  We shall use similar ideas and techniques as in proof of Lemma \ref{lemma:ik}.
  First, through a circular permutation, we can suppose, without restricting the generality, that 
  $ k =1 $.
 
 Let   $ a(1) =1$, $ P_1 = \{1, 2\} $ 
  and, inductively,
  $ P_{k+1} = P_k \cup \{ a(k+1), \pi( a(k+1))\} $,
  where
  $ a(k+1) = \min \{ t \in [ m ]:\  t \notin P_k \} $ (see also Example 2 below).
     \setlength{\unitlength}{.13cm}
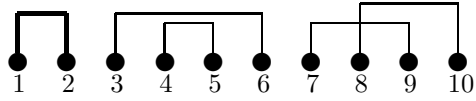
      \begin{equation*}
      \begin{picture}(20,11)
     
     \put(-20.5,0){1}
     
     \put(-20,3){\circle*{2}} 
     
     
     \put(-15.5,0 ){2}
     
     \put(-15,3){\circle*{2}}
     
     
     \put(-10.5,0){3}
     
     \put(-10,3){\circle*{2}}
     
     
     \put(-5.5,0){4}
     
     \put(-4.9,3){\circle*{2}}
     
     
     \put(-0.5,0){5}
     
     \put(0,3){\circle*{2}}
     
     
     \put(4.5,0){6}
     
     \put(5,3){\circle*{2}}
     
     
     \put(9.5,0){7}
     
     \put(10,3){\circle*{2}}
     
     
     \put(14.5,0){8}
     
      \put(15,3){\circle*{2}}
     
     
     \put(19.5,0){9}
     
     \put(20,3){\circle*{2}}
     
     
     \put(24,0){10}
     
     \put(25,3){\circle*{2}}
     
 
      
      \put(-10,3){\line(0,1){5}}
      
     \put(-10,8){\line(1,0){15}}
      
      \put(5,3){\line(0,1){5}}
      
      
      \put(-5,3){\line(0,1){4}}
      
      \put(-5,7){\line(1,0){5}}
      
      \put(0,3){\line(0,1){4}}
      
      
      \put(10,3){\line(0,1){4}}
      
      \put(10,7){\line(1,0){10}}
      
      \put(20,3){\line(0,1){4}}

     \put(25,3){\line(0,1){6}}
     
     \put(15,3){\line(0,1){6}}
     
     \put(25,9){\line(-1,0){10}}

      \linethickness{.55mm}

     
     \put(-20,4){\line(0,1){4}}

    \put(-15, 3){\line(0,1){5}}
     
     \put(-15, 8){\line(-1,0){5}}
      \end{picture}
       \end{equation*}
  
   \textbf{Example 3}. In the diagram above, 
   $P_1=\{ 1, 2\}$, $a(2) = 3 $, 
   $ P_2 = \{1, 2, 3, 6 \} $,
   $a(3) = 4 $, $ P_3 = \{ 1, 2, 3, 4, 5, 6\}$,
  $ a(4) = 7 $ and $ a(5)=8$. 
  
  \bigskip
  
  In particular, $ P_{ \frac{m}{2} } = [ m ] $ and for each
  $k > 1 $ we have that
  \begin{equation}
   \label{eq:pk}
   \{1, 2, \dots, a(k) -1 \} \subset P_k.
   \end{equation}
   
    As in the proof of Lemma \ref{lemma:ik}, for each
     $ k \leq \frac{m}{2} $,
     denote
        $ \displaystyle A_k =  \{ \oi[ P_k]:\  \oi \in A_{\frac{m}{2}} \} $.  
  Again,
           $ A_{\frac{m}{2}} = \{ \oi \in I(N, m):\ 
             v( \pi, \overrightarrow{\sigma}, \oi) \neq 0 \} $ and
   for each   
   $ k $ 
   and each
   $ \overrightarrow{\alpha} \in A_k $,
   property 
     $(\mathfrak{p}.1)$
     and relation 
     (\ref{eq:pk})
     give that
     \begin{align*}
     \#\big\{
     \oi[\{ a(k+1)\}]: \oi \in A_{ \frac{m}{2}}, \oi[ P_k] = \overrightarrow{\alpha}
     \big\} \leq N. 
     \end{align*}
  Also,  whenever
       $ \oi \in A_{ \frac{m}{2} } $,
       property
             $(\mathfrak{p}.4)$,
     gives that  
      $ \oi [ \{ \pi ( a(k+1))\}] $
      is uniquely determined by
      $ \oi [ \{ a(k+1)\}] $. 
      Therefore
      $
             \#( A_{k+1} ) \leq N \cdot \# ( A_k ),
             $
        which gives
             \[
                    \#( A_{\frac{m}{2}} ) \leq N^{ \frac{m}{2} -1} \cdot \# ( A_2 ).
                    \]

   On the other hand, 
   \begin{align*}
   A_2 = \big\{ ( i_1, i_{ -1}, i_2, i_{-2} ) \in [ N]^4: 
   \overrightarrow{i} [ \{ 1, 2\} ] = (i_1, i_1, i_{ -1}, i_2, i_{-2} ) )
   \textrm{ for some } 
   \overrightarrow{i} \in A_{ \frac{m}{2} }
   \big\}.
   \end{align*}  
  Since 
   $  A_{ \frac{m}{2}} \subseteq I ( N , m) $,
   we have that
   $ i_{ -1} = i_2 $, 
   hence
   \begin{align*}
      A_2 = & \big\{ ( i_1, i_2, i_2, i_{-2} ) \in [ N]^4: 
      \overrightarrow{i} [ \{ 1, 2\} ] = (i_1, i_1, i_{ -1}, i_2, i_{-2} ) )
      \textrm{ for some } 
      \overrightarrow{i} \in A_{ \frac{m}{2} }
      \big\}\\
       \subseteq &
       \big\{ ( i_1, i_2, i_2, i_{-2} ) \in [ N]^4: 
      \sigma_{1, N} (i_1, i_2) =
       t \circ \sigma_{2, N} ( i_2, i_{-2} )
             \big\},
      \end{align*} 
   which gives
   \begin{align*}
  \# (A_2) \leq \mathfrak{j}(\sigma_{1, k}: \sigma_{2, N}),
   \end{align*}
 therefore
    $ \# (A_{ \frac{m}{2} })   = o (N^{ \frac{m}{2} + 1}) $,
    which gives
     $ \displaystyle \mathcal{V}(\pi, \overrightarrow{\sigma} ) =
    o(1)$,
   and the conclusion follows. 
 \end{proof}
 
 \begin{lemma} \label{lemma:2:5}
 Let 
 $ \pi \in P_2(m) $ 
 be such that
  $ \pi(1) =2 $.
  If for all
   $ N $
   we have that
  $ \sigma_{1, N} = t \circ \sigma_{2, N} \circ t $,
  then
   \begin{align*}
   \mathcal{V} ( \pi , \overrightarrow{\sigma}) =
   \left\{
   \begin{array}{l l}
   1, & \textrm{ if } m =2\\
   \mathcal{V}( \pi^\prime, \overrightarrow{\sigma}^\prime),
    & \textrm{ if } m > 2.
   \end{array}
   \right.
   \end{align*}
   where 
   $ \pi^\prime $ 
   is obtained by removing the block 
   $ (1, 2) $
   from 
   $ \pi $,
   and
   $ \overrightarrow{\sigma}^\prime 
   = ( \sigma_3, \sigma_4, \dots, \sigma_m ) $.
 \end{lemma}
 
\begin{proof}
If
 $ m = 2 $, 
 then 
 $ \pi = (1, 2) $ 
 and
  \begin{align*}
  \mathcal{V}( \pi, \overrightarrow{\sigma})
  = \frac{1}{N}
  \sum_{i_1, i_2 = 1}^N
  E \big( g_{\sigma_{1, N}( i_1, i_2)} g_{ \sigma_{2, N}(i_2, i_1)}  \big)
  = \frac{1}{N}
   \sum_{i_1, i_2 = 1}^N
  \frac{1}{N} 
  \delta_{\sigma_{1, N}(i_1, i_2)}^{ t \circ \sigma_{2, N}\circ t
  (i_1, i_2)} = 1.
  \end{align*}
  
If
 $ m > 2 $, 
 remark that remark that for any
  $ \overrightarrow{i} = (i_1, i_{-1}, \dots, i_{m}, i_{-m} ) $
        from
        $ I(N, m) $
  we have that      
       $ \overrightarrow{i}^\prime 
          =
          (i_1, i_{-3}, i_4, i_{-4}, \dots, i_{m}, i_{-m})  $
     is an element of the set
     $ I(N, m-2)$.  
       Equations (\ref{v:V:1}) gives then
        \begin{align*} 
         v( \pi, \overrightarrow{\sigma}, \oi) =
         \frac{1}{N} 
         \delta_{i_1}^{i_{-2}} 
         \cdot 
         v( \pi^\prime, \overrightarrow{\sigma}^{\prime}, \oi^\prime),
       \end{align*}
        henceforth
          \begin{align*}
             \mathcal{V}( \pi, \overrightarrow{\sigma}) &=
            \frac{1}{N}
             \sum_{\oi \in I(N, m)}
             \frac{1}{N} \delta_{i_1}^{i_{-2}}
             v( \pi^\prime, \overrightarrow{\sigma}^{\prime}, \oi^\prime)\label{eq:25}\\
           & = 
             \frac{1}{N} \sum_{i_2 \in [ N ]}
           \mathcal{V}( \pi^\prime, \overrightarrow{\sigma}^{\prime} ) = 
             \mathcal{V}( \pi^\prime, \overrightarrow{\sigma}^{\prime} )\nonumber
             \end{align*}
       and the conclusion follows.
\end{proof}


 We shall use Lemmata \ref{lemma:ik}, \ref{lemma:ik1} and
 \ref{lemma:2:5} for the main results of this paper, in the next two sections.


\section{Main results}

\begin{lemma}\label{lemma:3:1}
 Suppose that for any positive integer  
 $ N $,
 $ G_N $ 
 is a 
 $ N \times n $
 Gaussian random matrix and that
 $ \mu_N $
  is a permutation from
  $ \mathcal{S}( [ N ]^2 ) $.
  If
  $ \mathfrak{j} ( \mu_N : \mu_N ) = o( N^2 ) $,
  then
  $ G_N^{ \mu_N } $ 
  is asymptotically circularly distributed of variance 1;
  more precisely
  \[
  \lim_{N \rightarrow \infty} \kappa_2 
  \big( G_N^{\mu_N}, (G_N^{\mu_N})^\ast\big) 
  =
  \lim_{N \rightarrow \infty} \kappa_2 
  \big( (G_N^{\mu_N})^\ast, G_N^{\mu_N} \big)
  =1
  \]
  and all other free cumulants of $ G_N^{\mu_N} $ 
  and $ (G_N^{\mu_N})^\ast $
  cancel asymptotically.
\end{lemma}

\begin{proof}

From the definition of the free cumulants, it suffice to show that,
  for any 
  $ m $ 
  positive integer and any 
 $ \epsilon: [ m ] \rightarrow \{ 1, \ast \} $,
 we have 
 \begin{align*}
 \lim_{ N \rightarrow \infty }
 E \circ \tr 
 \big( 
 (G_N^{\mu_N})^{ \epsilon(1)} (G_N^{\mu_N})^{ \epsilon(2)} 
 \cdots
 (G_N^{ \mu_N } )^{ \epsilon(m)} 
 \big) = 
 \# NC_2^\epsilon ( m ),
 \end{align*} 
  where
  $ NC_2^{\epsilon} ( m ) = \{ \pi \in NC_2 (m ): 
  \epsilon(k) \neq \epsilon(l) 
  \textrm{ whenever } ( k, l) \in \pi 
  \}$.
  
  Since for any 
  $ \sigma \in \mathcal{S}( [ N ]^2 )$, we have that
   $ ( G_N^\sigma )^\ast = G_N^{ t \circ \sigma \circ t } $,
  using equation (\ref{eq:wick}), it suffices to show that
  \begin{equation}
  \label{V:01}
  \lim_{N \rightarrow \infty} 
  \mathcal{V}( \pi , \overrightarrow{\sigma_N}) = \left\{
  \begin{array}{c l}
  0, & \textrm{ if } \overrightarrow{\sigma_N} \in P_2(m) \setminus NC_2^\epsilon ( m ) \\
  1, & \textrm{ if } \overrightarrow{\sigma_N}
   \in NC_2^\epsilon ( m ).
  \end{array}
  \right.
  \end{equation}
where
$ \overrightarrow{ \sigma_N} = 
( \sigma_{1,  N}, \sigma_{2, N} \dots, \sigma_{m, N} ) $
with
$ \sigma_{k, N } = 
\left\{ \begin{array}{c l}
  \mu_N, & \textrm{ if } \epsilon ( k ) = 1 \\
  t \circ \mu_N \circ t , & \textrm{ if } \epsilon ( k ) = \ast.
  \end{array}
\right. $

If
 $ \pi \notin NC_2(m) $,
  then (\ref{V:01}) follows from Lemma \ref{lemma:ik}. 
If
 $ \pi \in NC_2( m ) $,
 we can suppose without loss of generality that
 $ \pi $ 
 is the juxtaposition of 
 $ (1, 2) $ 
 and some pairing 
 $ \pi ^ \prime $
 from
 $ NC_2( m - 2 ) $.
 If 
 $ \epsilon (1) = \epsilon (2) $,
 then  the condition
  $ \mathfrak{j} ( \mu_N : \mu_N ) = o(N^2) $
  and  Lemma \ref{lemma:ik1} gives that
   $ \displaystyle \lim_{ N \rightarrow \infty }
   \mathcal{V} ( \pi , \overrightarrow{\sigma_N} )
    = 0 $.
  If
  $ \epsilon(1) \neq \epsilon(2) $, 
  then the conclusion follows from Lemma \ref{lemma:2:5} 
  and an inductive argument on 
  $ m $.
\end{proof}


 \begin{thm}\label{thm:1}
 Suppose that
  $ M_1, M_2 $
   are positive integers and  that, for each positive integer 
   $ N $,
 $ \big\{ \mu_{k, N}: \ 1 \leq  k \leq  M_1+ M_2  \big\} $
 is a family of permutations from 
 $ \mathcal{S}[([N]^2) $ 
 such that
 \begin{itemize}
\item[(i)]$ \mu_{k, N} $ 
is symmetric for each 
$ k \leq M_1 $, 
and
$ \mathfrak{j}( \mu_{k, N}: \mu_{ k, N} ) = o(N^2) $
for each
$ M_1 < k < M_1 + M_2  $
\item[(ii)] whenever
 $ a \neq b $, 
 we have that
 \[
 \# \big\{ (i, j, k) \in [ N ]^3:\ 
  \mu_{a, N} ( i, j) \in \{
  \mu_{b, N} (i, k),
   \mu_{ b, N }( k, j ), 
    t \circ \mu_{b, N} \circ t  ( i, k ),
    t \circ \mu_{b, N} \circ t ( k, j)
  \}
 \big\}
  \]
  is 
  $ o(N^2) $.
 \end{itemize}

    If 
    $ G_N $ 
    is an
     $ N \times N $
     Gaussian random matrix, then the limit in distribution of the family 
  $\big\{ G_N^{\mu_{k, N}}: k \in [ M ] \big\}$
   is a free family of
    $ M_1 $  
    semicircular and
    $ M_2 $
    circular non-commutative random variables of variance $ 1 $.
\end{thm}
\begin{proof}
Condition (i) gives that, for
$ 1 \leq k \leq M_1 $,
 each 
$ G_N^{\mu_{ k, N}} $
is selfadjoint and asymptotically semicircular of variance 
$ 1 $,
and, for
$ M_1 + 1 \leq k \leq M_2 $,
according to Lemma \ref{lemma:3:1},
each
$ G_N^{ \mu_{k, N}} $
is asymptotically circular of variance 
$ 1 $.

Fix 
 $ m $ 
 a positive integer;
 let 
 $ \mathfrak{f} : [ m ] \rightarrow [ M_1 + M_2 ] $
 be a given map and let
 $ \epsilon: [ m ] \rightarrow \{ 1, \ast \} $
 be a map such that
$ \epsilon( s ) = 1 $ 
whenever
  $ l ( s ) \leq M_1 $.
  
  It suffices though to show that in the moment-free cumulant decomposition of 
   \[
    E \circ \tr \big( ( G_N^{\mu_{ \mathfrak{f}(1), N}})^{ \epsilon(1)}  
   ( G_N^{\mu_{ \mathfrak{f} ( 2 ), N}})^{ \epsilon( 2 )}   
   \cdots
   ( G_N^{\mu_{ \mathfrak{f} ( m ), N}})^{ \epsilon(m)} 
          \big) \] 
 all the free cumulants vanish asymptotically, except for the ones either of the form
  $ \kappa_2 (  G_N^{ \mu_{k, N}}, G_N^{ \mu_{k, N}} ) $
 with
  $ k \leq M_1 $, 
  or of one of the forms
  $ \kappa_2 (  G_N^{ \mu_{k, N}}, (G_N^{ \mu_{k, N}})^\ast ) $,
 $ \kappa_2 ( (G_N^{ \mu_{k, N}})^\ast,  G_N^{ \mu_{k, N}},  )$
 with
  $ M_1 < k \leq M_1 + M_2 $.
  
  Using condition (i) and the fact that 
  $  ( G_N^\sigma )^\ast = G_N^{ t \circ \sigma\circ t } $,
   an equivalent form of the statement above is 
  \[
  \lim_{ N \rightarrow \infty }
   E \circ \tr \big(  G_N^{\sigma_{ 1 , N}}
      G_N^{\sigma_{  2 , N}} 
     \cdots
     G_N^{\sigma_{ m , N}}
            \big) 
            =
            \# NC_2^{\mathfrak{f}, \epsilon} ( m )
  \]
  where
 \begin{align*}
 NC_2^{\mathfrak{f}, \epsilon} ( m ) = \{ \pi \in NC_2 ( m ): 
 \mathfrak{f}(s) =  \mathfrak{f}( \pi (s) ) \textrm{  for all } 
  s \in [ m ]
 &  \textrm { and }  \epsilon(s) \neq \epsilon ( \pi(s) )\\
  & \textrm{ whenever } \mathfrak{f} ( s ) > M_1 \},
 \end{align*}
 and
 \begin{align*}
 \sigma_{k, N} = \left\{ 
 \begin{array}{l l}
 \mu_{\mathfrak{f}(k), N}, & \textrm{ if } \epsilon(k) =1\\
 t \circ \mu_{\mathfrak{f}(k), N} \circ t, & \textrm{ if }
 \epsilon(k) = \ast.
 \end{array} \right.
 \end{align*}
 On the other hand, from Lemma \ref{lemma:ik}, we have 
 \begin{align*}
  \lim_{ N \rightarrow \infty }
    E \circ \tr \big(  G_N^{\sigma_{ 1 , N}}  
       G_N^{\sigma_{  2 , N}}  
      \cdots
       G_N^{\sigma_{  m , N}}
             \big) 
             =
         \sum_{ \pi \in NC_2 ( m ) }
             \lim_{N \rightarrow \infty}
             \mathcal{V} ( \pi, \overrightarrow{ \sigma_N} ).
 \end{align*}
  So it suffices to show that 
  \begin{align}\label{main:V}
  \lim_{N \rightarrow \infty }
  \mathcal{V} ( \pi, \overrightarrow{ \sigma_N} )
  =
  \left\{
  \begin{array}{l l}
  1,  & \textrm{ if } \pi \in NC_2^{ \mathfrak{f}, \epsilon } (  m )\\
  0, & \textrm{ otherwise}.
  \end{array}
  \right.
  \end{align}
  
  We shall show (\ref{main:V}) using the same technique as in the proof of Lemma \ref{lemma:3:1}. Without loss of generality, we can suppose (via a circular permutation) that 
  $ \pi $ 
  is the juxtaposition of the block
  $ (1, 2) $
  and some pairing 
  $ \pi^\prime $
  from
   $ NC_2 ( m - 2 ) $.
   
  For any 
  $ \sigma, \tau \in \mathcal{S} ( [ N ]^2 ) $,
   the definition of
   $ \mathfrak{j}( \sigma : \mu ) $, 
   gives that
   \begin{align}\label{cond:j}
   \#\big\{ (i, j, k):\ 
   \sigma(i, j) \in
    \{
    \mu( i, k), \mu ( k, j), 
     t \circ & \mu \circ t (i, k),
     t \circ \mu \circ t ( k, j) 
   \}
   \big\}\\
   & = \frac{1}{2}\sum_{ 
    \substack{\tau_1 \in \{ \sigma, t \circ \sigma \circ t \}\\ \tau_2 \in \{ \mu, t \circ \mu \circ t   \}} 
    }
    \mathfrak{j} ( \tau_1 : \tau_2 ) 
    + \mathfrak{j} ( \tau_2 : \tau_1 ).\nonumber
   \end{align}
   In particular, if the left hand side of equation (\ref{cond:j}) is
    $ o( N^2 )$,
    so is each of the eight terms in the right hand side.
  So condition (ii) gives that 
  $ \mathfrak{j} (\sigma_{1, N}: \sigma_{2, N} ) = o(N^2) $
  whenever
 $ \mu_{\mathfrak{f}( 1 ) , N} \neq \mu_{\mathfrak{f} ( 2 ), N} $.
  Henceforth, utilizing Lemma \ref{lemma:ik1}, we have that
   $ \displaystyle \lim_{N \rightarrow \infty }
   \mathcal{V} ( \pi, \overrightarrow{\sigma_N } ) 
   = 0 $
   unless 
   $ \mu_{\mathfrak{f}( 1 ) , N} = \mu_{\mathfrak{f} ( 2 ), N} $.  
   
   Suppose that 
   $ \mu_{\mathfrak{f}( 1 ) , N} = \mu_{\mathfrak{f} ( 2 ), N} $. 
   If 
   $ \mathfrak{f}(1) \leq M_1 $, 
   then 
   $ \epsilon(1) = \epsilon(2) = 1 $
   and, from condition (i) we have that
  $ \mu_{\mathfrak{f}( 1 ), N} = 
   t \circ \mu_{\mathfrak{f} ( 2 ), N} \circ t $
 so (\ref{main:V}) follows from Lemma \ref{lemma:2:5} and an inductive argument on
 $ m $. 
  Suppose then that
  $ \mathfrak{f}( 1 ) > M_1 $.
  If
  $ \epsilon (1) = \epsilon ( 2 ) $, 
 condition (i) gives that
  \[
  \mathfrak{j} ( \sigma_{1, N} : \sigma_{2, N} ) 
  = \mathfrak{j} ( \mu_{ \mathfrak{f}(1), N} :
   \mu_{\mathfrak{f} ( 1 ), N} )
    = \mathfrak{j} (  t \circ \mu_{ \mathfrak{f}(1), N} \circ t :
      t \circ \mu_{\mathfrak{f} ( 1 ), N} \circ t  )
      = o( N^2 )
 \]
 so
  $ \displaystyle \lim_{N \rightarrow \infty }
    \mathcal{V} ( \pi, \overrightarrow{\sigma_N } ) 
    = 0 $. 
    If
    $ \epsilon (1) \neq \epsilon (2) $, 
    then
    $\mu_{\mathfrak{f}( 1 ), N} = 
           t \circ \mu_{\mathfrak{f} ( 2 ), N} \circ t $
           follows  again from condition (i)
  and, as above, the conclusion follows using Lemma \ref{lemma:2:5}.  
\end{proof}

In particular, equation (\ref{cond:j}) gives that condition (ii) from Theorem \ref{thm:1} is symmetric in 
$ a  $ 
and
$ b $ 
(i.e. does not need to be verified for both pairs
$ (a, b) $ and $ (b, a) $).

\begin{remark}\label{rem:ii}
If in Theorem \ref{thm:1} above we have that 
$ \big( \mu_{b, N} \big)_N $
 are symmetric, condition \emph{(ii)} becomes:
 \begin{itemize}
 \item[(1)] 
 $\# \big\{ (i, j, k) \in [ N ]^3:\ \mu_{a, N} (i, j) = \mu_{b, N}(i, k) 
  \big\} = o(N^2) $
  whenever 
  $ a \neq b $ 
  and
    $ \mu_{a, N} $
     is symmetric
     \item[(2)] $\# \big\{ (i, j, k) \in [ N ]^3:\ \mu_{a, N} (i, j)\in
       \{ \mu_{b, N}(i, k) , \mu_{b, N}( k, j) \}
        \big\} = o(N^2) $
        whenever 
        $ a \neq b $ 
        and
          $ \mathfrak{j}( \mu_{ a, N} : \mu_{a, N}) = o(N^2 ) $.
 \end{itemize}
\end{remark}

For the particular case of two permutations, one being the identity,  Theorem \ref{thm:1} and Remark \ref{rem:ii} have the following consequence.

\begin{cor}\label{cor:id}
Suppose that for each positive integer 
$ N $, 
 $ G_N $ 
 is an $ N \times N $ 
 Gaussian random matrix.
 If 
   $ \big( \sigma_N \big)_N $  
   is a sequence of permutations with each 
   $ \sigma_N $ 
  from 
   $ \mathcal{S}( [ N ]^2 ) $
   such that the following conditions are satisfied
   \begin{enumerate}
   \item[(i)] either  all 
   $ \sigma_N $
   are symmetric, or all 
   $ \sigma_N $
   satisfy
   $ \mathfrak{j} ( \sigma_N : \sigma_N ) = o(N^2)  $;
   \item[(ii)]
   $ \displaystyle
    \# \big\{ ( i, j, k) \in [ N ]^3 :\  \sigma_N (i, j) \in \{ 
    (i,k), ( k , j) \}
    \big\} 
    = o( N^2 )
    $
   \end{enumerate}
   then 
   $ G_N $ 
   and 
   $ G_N ^{ \sigma_N } $
   are asymptotically free.
 \end{cor}

To simplify the statement of the next result, we shall introduce the following notation. For 
$ \phi, \psi $ 
to permutations from 
$ \mathcal {S} ( [ N ]) $,
denote by
$ \phi \otimes \psi  $
the permutation in 
$ \mathcal{S}([ N ]^2 ) $
given by
$ \phi \otimes \psi (i, j) = ( \phi(i), \psi(j) ) $
and denote by
 $ \mathfrak{c} ( \phi, \psi) $
 the number of fixed point of 
 $ \phi^{-1} \psi $
( i.e.  $
\mathfrak{c}(\phi , \psi) = \#\{ i \in [ N ]:\ \phi (i) = \psi (i)\} $ ).

\begin{thm}\label{thm:tensor}
Let 
$ M_1 $
and
$ M_2 $
be two positive integers and suppose that for each positive integer
 $ N $, 
$\{ \phi_{k, N}, \psi_{k, N}: k \in [ M_1 + M_2 ]\} $
is a family of permutations from 
$ \mathcal{S}([ N ]) $
such that:
\begin{itemize}
\item[(i)]
$  \phi_{k, N} = \psi_{k, N} $
whenever
$ k \leq M_1 $
and
$ \mathfrak{c} ( \phi_{ k N }, \psi_{ k, N } ) = o( N ) $
whenever
$ M_1 \leq k < M_1 + M_2 $.
\item[(ii)] if
$ k \neq p $, 
then
 $ \mathfrak{c} ( \phi_{k, N}, \phi_{p, N})
  + \mathfrak{c} ( \phi_{ k , N }, \psi_{p, N })
  + \mathfrak{c} ( \psi_{k, N}, \psi_{p, N}) = o(N).
  $
\end{itemize}
Then the random matrices
 $ \{ G_N^{ \phi_{k, N} \otimes \psi_{ k, N} }:
  k \in [ M_1 + M_2 ] \} $
 form an asymptotically free family of semicircular, if 
 $ k \leq M_1 $,
 respectively circular, if
  $ M_1 < k \leq M_1 + M_2 $,
  random matrices.
  Moreover, the family 
  $ \big\{
  G_N^{ \phi_{k, N} \otimes \psi_{ k, N} }:
   k \in [ M _1 + M_2 ] 
    \big\} $
    is asymptotically free from its transpose.
\end{thm}

\begin{proof}
First, note that for any
$ \phi, \psi \in \mathcal{S} ( [ N ]^2 ) $,
 $ t \circ   ( \psi \otimes \psi ) \circ t
 = \psi \otimes \phi $.
 
 Next, note that if
  $ \phi_1, \phi_2, \psi_1, \psi_2  \in \mathcal{S} ( [ N ]) $, then
  \begin{align}\label{j:tensor}
  \mathfrak{j} ( \phi_1 \otimes \psi_1 : \phi_2\otimes \psi_2) =&
  \# \big\{
  (i, j, k):\  \big(\phi_1 ( i) , \psi_1(j)  \big) =
  \big(
  \psi_2 (k), \phi_2( j )
  \big) 
  \big\}\\
  = &
  N \cdot 
  \mathfrak{c} ( \psi_1, \phi_2 ).\nonumber
  \end{align}
If 
 $ \phi_{k, N} = \psi_{k, N} $,
 then
 $ \mu_{ k, N} $ 
 is symmetric, so 
 $ G_N^{ \mu_{k, N}} $ 
 is asymptotically semicircular.
 If 
  $ \mathfrak{c} ( \phi_{ k N }, \psi_{ k, N } ) = o( N ) $,
  then (\ref{j:tensor}) gives that 
 $
  \mathfrak{j}( \mu_{k, N}: \mu_{k, N})= o(N^2)$,
and Lemma \ref{lemma:3:1} implies that
 $ G_N^{ \mu_{ k, N}} $
 is asymptotically circular.
 
 Furthermore, condition (ii) and equation (\ref{j:tensor}) imply that the family of permutations
 $ \big\{ \phi_{ k, N} \otimes \psi_{ k, N }: k \in [ M_1 + M_2 ] 
 \big\} $
 satisfies condition (ii) of Theorem \ref{thm:1}, hence the asymptotic free independence of 
 $ \{ G_N^{ \phi_{k, N} \otimes \psi_{ k, N} }:
   k \in [ M_1 + M_2 ] \} $.
   
   For the last part, the asymptotic freeness from transposes, it suffices to show that condition (ii) of Theorem \ref{thm:1} is satisfied by any pair 
   $ \phi_{k, N} \otimes \psi_{ k, N} $ , 
   $ t \circ  ( \phi_{ p, N} \otimes \psi_{ p, N } ) $
   with
   $ k, p \in [ M_1 + M_2 ] $. 
   To show the last statement, it suffices to proof that each of the terms in the right hand side of equation (\ref{cond:j}) are
   $ o(N^2) $ 
   for 
   $ \sigma = \phi_{k, N} \otimes \psi_{ k, N} $
   and 
   $ \mu = t \circ  ( \phi_{ p, N} \otimes \psi_{ p, N } ) $.
   Moreover, since
    $ \mathfrak{j}( \sigma: \mu ) =
     \mathfrak{j} ( t \circ \mu \circ t : t \circ \sigma \circ t ) $
     and 
     $ \mathfrak{j}( \sigma: \mu ) = \mathfrak{j} ( t \circ \sigma: t \circ \mu ) $, 
     it suffices to show that  for any
         $ \phi_1, \phi_2, \psi_1, \psi_2 \in \mathcal{S}([ N ]) $, 
         we have 
     \begin{align}\label{eq:tens:2}
 \mathfrak{j} ( \phi_1 \otimes \psi_1 :
 t \circ ( \phi_2 \otimes \psi_2 ) )
= \mathfrak{j}( \phi_1 \otimes \psi_1 :
( \phi_2 \otimes \psi_2 ) \circ t  ) = N.
     \end{align}
Indeed,   
   \begin{align*}
\mathfrak{j} ( \phi_1 \otimes \psi_1: t \circ ( \phi_2 \otimes \psi_2 ) ) &
 = 
\# \big\{ (i, j, l):\ \big( \phi_1(i), \psi_1(j) \big) =
\big( \phi_2 (j), \psi_2 (l) \big)
\big\}\\
& = 
\# \big\{ (i, j, l):\ j = \phi_2^{-1} \phi_1 (i), \ 
l = \psi_2^{-1}\psi_1 (j) 
\big\}\\
& = N,
 \end{align*}
 and
    \begin{align*}
 \mathfrak{j} ( \phi_1 \otimes \psi_1: ( \phi_2 \otimes \psi_2 )
 \circ t  ) &
  = 
 \# \big\{ (i, j, l):\ \big( \phi_1(i), \psi_1(j) \big) =
 \big( \psi_2 (j), \phi_2 (l) \big)
 \big\}\\
 & = 
 \# \big\{ (i, j, l):\ j = \psi_2^{-1} \phi_1 (i), \ 
 l = \phi_2^{-1}\psi_1 (j) 
 \big\}\\
 & = N.
  \end{align*}
\end{proof}

\begin{cor}\label{cor:id}
Suppose that for each positive integer 
$ N $, 
 $ G_N $ 
 is an $ N \times N $ 
 Gaussian random matrix.
 \begin{itemize}
  \item[1.] If 
  $ \big( \phi_N \big)_N $ 
  is a sequence of permutations such that each
   $ \phi_N $ 
  is an element
   $ \mathcal{S} ( [ N ]) $
  with 
  $ o(N ) $ 
  fixed points, then
  $ G_N $ , $ G_N^{ \Phi_N \otimes \phi_N } $
  and their transposes form an asymptotically free family.
  \item[2.]If 
    $ \big( \phi_N \big)_N $ ,
    $ \big( \psi_N \big)_N $
    are two sequences of permutations such that for each 
    $ N $,
     $ \phi_N $, 
     $ \psi_N $
     and
     $ \phi^{-1}_N \psi_N $
     are elements of 
     $ \mathcal{S}([ N ]) $
     with
      $ o( N ) $
      fixed points, then
      $ G_N $, 
      $ G_N^{ \phi_N \otimes \psi_N } $
      and their transposes form an asymptotically free family.
 \end{itemize}
\end{cor}

 Two transforms that appear in literature and are given by entry permutations are the partial transpose, relevant in Quantum Information Theory (\cite{horod}, \cite{aubrun}) and the ``mixing map'' appearing in Physics literature (see \cite{zycz}, \cite{billiard}).

Following \cite{mingo-popa-wishart}, \cite{mingo-popa-wishart2}, we will define the 
$ N \times N $ 
partial transpose 
$ \Gamma_ N $
as below. First consider the bijection 
 $ \varphi : [ N^2 ]^2  \rightarrow [ N ]^4 $
given by
 $ \varphi (i, j) = ( a, b, c, d ) $
 whenever
 \begin{align*}
  (i, j) = \big( ( a-1)N + b, (c-1)N + d \big). 
\end{align*}
Then take 
$ \gamma: [ N ]^4 \rightarrow [ N ]^4 $
given by
$ \gamma ( a, b, c, d) = ( a, d, c, b )$. The $ N \times N $ partial transpose is the map
 \begin{align*}
 \Gamma_N = \varphi^{-1} \gamma \varphi 
 \end{align*}
 
 Intuitively (see \cite{aubrun}, \cite{banica-nechita}), we see a 
 $ N^2 \times N^2 $ 
 matrix as a 
 $ N \times N $ 
 block-matrix, each entry being a 
 $ N \times N $
  matrix. Then 
   $ \Gamma_N $
   is obtained by transposing each block, but keeping the positions of the blocks.
   
   With the notations from above, the ``mixing map''
    $ \mu_N $ 
    is defined via 
    \begin{align*}
     M_N =  \varphi^{-1} \mu \varphi  
     \end{align*}
    with
    $ \mu: [ N ]^4 \rightarrow [ N ]^4 $
    given by
    $ \mu(a, b, c, d) = (a, c, b, d) $.
    
  A consequence of Theorem \ref{thm:1} and   Corollary \ref{cor:id} is the following.
    
    \begin{cor}\label{cor:trio}
  The random matrices
   $ G_{ N^2 } $,
   $ G_{ N^2 }^{ \Gamma_N } $
   and
   $ G_{ N ^2 }^{ M_N } $
   form an asymptotically free family (as 
   $ N \rightarrow \infty $).
    \end{cor}
    
    \begin{proof}
    Let 
    $ \widetilde{t} : [ N ]^4 \rightarrow [ N ]^4 $
    be given b,
     $ \widetilde{t} ( a, b, c, d) = (c, d, a, b ) $, 
     that is
     $ t = \varphi^{-1} \widetilde{t} \varphi $.
     
    Note that 
    $ \Gamma_N $ 
    is symmetric, since
     $ \gamma = 
      \widetilde{t} \gamma \widetilde{t} $ 
      and that
    \begin{align*}
    \mathfrak{j} ( M_N: M_N ) & = 
    \# \big\{ (i, j, k)\in [ N^2 ]^3 : M_N (i, j) = t\circ M_N \circ t (k, j)
    \big\}\\
    = & \#\big\{(a, b, c, d, e, f) \in [ N ]^6 :
    \mu (a, b, c, d ) = 
    \widetilde{t} \mu \widetilde{t} (e, f, c, d)
    \big\}\\
    = & \#\big\{(a, b, c, d, e, f) \in [ N ]^6 :
    (a, c, b, d ) = (d, f, c, e ) \big\}\\
    = & N^3 = o\big((N^2)^2 \big).
    \end{align*}
    So condition (i) of Theorem \ref{thm:1} is satisfied. It suffices then to show that the identity, 
    $ \Gamma_N $  
    and 
    $ M_N $
    satisfy condition (ii) of Theorem \ref{thm:1}.
    And indeed
    \begin{align*}
    \# \big\{ (i, j, k) \in [ N^2]^3: \ & \Gamma_N (i, j) = (i, k)
    \big\}\\
    &= 
    \# \big\{ (a, b, c, d, e, f) \in [ N ]^6:\ (a, d, c, b) = (e, f, c, d )
    \big\}
    = N^3
    \end{align*}
    \begin{align*}
        \# \big\{  & (i, j, k)  \in [ N^2]^3: \  M_N (i, j) \in
        \{ (i, k), (k, j) \}
        \big\}\\
        &= 
        \# \big\{ (a, b, c, d, e, f) \in [ N ]^6:\ 
        (a, c, b, d) \in 
        \{  (e, f, c, d ), ( a, b, e, f ) \}
        \big\}
        = O(N^3 )
        \end{align*}
        and
       \begin{align*}
              \# \big\{  & (i, j, k)  \in [ N^2]^3: \  M_N (i, j) \in
              \{ \Gamma_N(i, k), \Gamma_N(k, j) \}
              \big\}\\
              &= 
              \# \big\{ (a, b, c, d, e, f) \in [ N ]^6:\ 
              (a, c, b, d) \in 
              \{  (a, f, e, b ), ( e, d, c, f ) \}
              \big\}
              = O(N^3 ),
              \end{align*}  
     hence the conclusion follows from Remark \ref{rem:ii}.         
    \end{proof}
    
    \section{Remarks}
    
    This last section contains some comments, mostly on the conditions from the statement of Theorem \ref{thm:1}.
    \begin{remark}\label{remark:4:1}
    If
     $ (\mu_{ 1, N })_N $ 
     and 
     $ ( \mu_{ 2, N })_N $
     are two family of permutations, each
      $ \mu_{k, N}$
       from 
      $ \mathcal{S}([ N ]^2) $,
     satisfying condition \emph{(ii)} but not condition \emph{(i)} 
     of Theorem \ref{thm:1}, then
     $ G_N^{\mu_{1, N}} $ 
     and
     $ G_N^{\mu_{2, N} } $
     are not necessarily asymptotically free,  nor are they necessarily asymptotically free from matrices with constant coefficients. 
    \end{remark}
    
    \begin{proof}
  We shall construct an example of such 
    $ \{\mu_{k, N} \}_{ N \in \mathbb{N} } $, $ k = 1, 2$.
    It suffices to specify 
    $\mu_{k, 2N } $ 
     and take 
     $ \mu_{k, 2N + 1 }(i, j)= \mu_{k, 2N }(i,j) $
     for 
     $ i, j \leq 2N $.
     
    Remark that
     $ G_{2N} = \left(
      \begin{array}{cc} G_{ 1, N } & X_N \\ X_N^\ast & G_{ 2, N }  \end{array}\right) $     
    where
    $ G_{ 1, N }, G_{2, N}, X_N $ 
    are
    $ N \times N $ 
    random matrices,
    $ G_{1, N} $ 
    and
    $ G_{2, N} $ 
    are Gaussian,
    $ X_N $ 
    is Ginibre,  upper-diagonal entries of 
    $ G_{1, N} $,
    and
     $ G_{2, N} $
     and  entries of 
     $ X_N $
     form an independent family of complex (off the diagonals of 
     $ G_{1, N} $, $ G_{2, N} $) and real (on the diagonals of
      $ G_{1, N}$, $ G_{2, N} $) Gaussian distributed random variables of mean
      $ 0 $
      and variance
      $ \frac{1}{\sqrt{2N} } $.
      
      Let
       $ \varphi_N : [ 2N] \rightarrow [ N ] $ 
       given by
       $ \varphi(k)= p $
        if and only if 
        $ k \equiv  p $ (mod $ N $). 
        Let
         $ \omega_N  \in \mathcal{S}([ N ]^2) $,
respectively
        $ \gamma_N \in \mathcal{S}( [ 2N ]^2 ) $
        be given by
        \begin{align*}
        \omega_N (i, j) = &
         \big(
        \varphi_N (i+1), \varphi_N (j +2 ) 
        \big)\\
        \gamma_N ( i, j) = &
        \big( i - \varphi_N (i ) + \varphi_N ( j ), 
        j - \varphi_N (j) + \varphi_N (i) 
        \big)
        \end{align*}
   
        Next, we define  the permutations
 $ \mu_{1, 2N}, \mu_{2, 2N}  \in
 \mathcal{S}([ 2N ]^2 ) $ 
 via
 \begin{align*}
 \mu_{1, 2 N} (i, j) & =
 \left\{ 
            \begin{array}{ll}
            \omega_N (i, j) & \textrm{ if }  i, j \leq N \\
            (i, j)  & \textrm{ otherwise }
            \end{array}
            \right.\\
    \mu_{2, 2N} (i, j) & = \gamma_N \circ \mu_{1, 2N} .       
 \end{align*}
I.e.
       $ G_{2N}^{\mu_{1, 2N}} = \left(
      \begin{array}{cc}
       G_{1, N}^{\omega_{ N} } & X_N \\ 
      X_N^\ast  & G_{2, N} 
     \end{array}\right) $
     and
     $  G_{2N}^{\mu_{2,  2N}} = \left(
        \begin{array}{cc}
       (  G_{1, N}^{\omega_{ N} } ) ^t&  X_N^t\\ 
       ( X_N^\ast)^t  &  G_{2, N}^t  
       \end{array}\right) $.
       
From Corollary \ref{cor:id}, 
$ G_{1, N}^{\omega_N} $ 
is asymptotically circular and free from its transpose.
Furthermore,  since 
       $ G_{2, N} $
       and
       $ X_N $
       are unitarily invariant with independent entries, 
 using the result from \cite{mingo-popa-transpose}, it follows that      
   the ensembles 
    $ \{  G_{2, N} \}_N  $,
     $ \{ G_{ 2,  N}^t \}_N $ ,
     $ \{ X_N\}_N $, $ \{ X_N^t \}_N $
     and 
    $ \{ G_{1, N}^{ \omega_{ N}},  ( G_{1, N}^{ \omega_{ N}} )^t \}_N $
    are asymptotically free.
    
      Hence the asymptotic joint distribution of
       $ G_{2N}^{\mu_{N, 1}} $ 
       and 
       $ G_{2N}^{\mu_{N, 2} } $ 
       is the joint distribution with respect to  
    $ \phi \circ \tr $
    of 
        $ A =\left(\begin{array}{c c}
          c_1 & c_2 \\
          c_2^\ast & s_1
          \end{array} \right) $ 
          and
           $ B = \left(  \begin{array}{c c}
            c_3 & c_4 \\
            c_4^\ast & s_2 
            \end{array}\right)$,
            where 
     $ \{ s_1, s_2, c_1, \dots, c_4 \} $
        is a free family from some noncommutative probability space 
           $ ( \mathcal{A}, \phi ) $,
       with 
            $ s_1, s_2 $, 
           respectively 
           $ c_1, c_2, c_3, c_4 $ 
  semicircular, respectively circular, of mean $ 0 $ and variance
            $\displaystyle \frac{1}{\sqrt{2} } $.   
                  
 Therefore, we have that 
 \begin{align*}
 \phi\circ \tr ( A^2) = 
 \phi\circ \tr ( B^2)  = 
             \frac{1}{2}
            \phi \big(
             c_1^2 +c_2 c_2^\ast + c_2^\ast c_2  + s_1^2 
            \big) =\frac{3}{4},
 \end{align*}
       while 
              $ \phi\circ \tr ( A^2 \cdot B^2 ) $
              equals  
          \begin{align*}
   \phi\circ \tr &
      \left(
       \left(\begin{array}{c c}
              c_1^2 + c_2c_2^\ast &
               c_1c_2  + c_2 s_1\\
              c_2^\ast c_1 + s_1 c_2^\ast  &  c_2^\ast c_2 + s_1^2
              \end{array} \right)
              \cdot
        \left(\begin{array}{c c}
                 c_3^2 + c_4c_4^\ast & c_3c_4  + c_4 s_2\\
                 c_4^\ast c_3 + s_2 c_4^\ast  &  c_4^\ast c_4 + s_2^2
                 \end{array} \right)       
      \right)\\
      =  &  \frac{1}{2}
       \phi   \big(
      ( c_1^2 + c_2c_2^\ast ) \cdot
      (c_3^2 + c_4c_4^\ast ) 
       + 
       ( c_1 c_2 + c_2 s_1 ) 
       \cdot
       ( c_3 c_4 + c_4 s_2 ) 
       \\
      &
      \hspace{1cm}  +  (c_2^\ast c_1 + s_1 c_2^\ast ) 
       \cdot 
       (c_3 c_4 + c_4 s_2 ) 
       + 
       ( c_2^\ast c_2 + s_1^2 )
       \cdot
       ( c_4^\ast c_4 + s_2^2 )
      \big)
        =  \frac{5}{8}.
           \end{align*}  
Henceforth 
$\phi(A^2 B^2) \neq  \phi(A^2)\cdot \phi(B^2)  $,
  so  $ A $ and $ B $ are not free.
  
   For the second part of the statement, denote by 
   $I_N $
    the 
    $ N \times N $
 identity matrix, and consider the 
  $ 2 N \times 2 N $
     matrices with constant coefficients
      $Z_N = \left( \begin{array}{c c } 
           I_N & 2 I_N \\
           0 & I_N 
           \end{array}\right) 
           $ 
           and 
          $ T_N  = \left( 
           \begin{array} { c c }
           I_N & - I_N \\
           2 I_N & - I_N
           \end{array}
           \right).
           $
   Then the asymptotic joint  distribution of 
   $ T_N, Z_N $
   and  
  $ G_{2N}^{ \mu_{1, N}} $ 
  is the joint distribution with respect to
  $ \phi \circ t $
  of the matrices   
    $ Z  = \left( 
              \begin{array}{ c c}
              1 & 2 \\ 0 & 1 
              \end{array}
              \right)$, 
              $ T = \left(
              \begin{array} { c c }
              1 & -1 \\
               2 & -1 
              \end{array}
               \right) $
 and
 $ A $
 from above.  
 
   If 
   $ A $
   were free from 
   $ \{ Z, T \} $, 
   then, using that 
       $ \tr \circ \phi (A) = \tr (T) = 0 $, 
       we would have that
       \begin{align*}
  \phi \circ \tr ( A Z A T ) 
  =  &
   \phi \circ \tr ( A ( Z - \tr (Z ) ) A  T ) + 
   \tr (Z) \phi \circ \tr ( A^2 T )\\ 
   = & 
   \tr(Z) \phi\circ \tr (A^2) \tr (T) = 0.
       \end{align*}
   But, since
    $\{ s_1, c_1, c_2\} $ 
    form a free family,
                     \begin{align*}
   \phi \circ & \tr (A Z A T )  =
                     \tr \circ \phi \left ( 
                     \left(\begin{array}{c c}
                     c_1 & c_2 \\
                     c_2^\ast & s_1
                     \end{array} \right)
                     \cdot 
                      \left( 
                     \begin{array}{ c c}
                     1 & 2 \\ 0 & 1 
                     \end{array}
                     \right)
                     \cdot
                     \left(\begin{array}{c c}
                     c_1 & c_2 \\
                     c_2^\ast & s_1
                     \end{array} \right)
                     \cdot
                     \left(
                     \begin{array} { c c }
                     1 & -1 \\
                     2 & -1 
                     \end{array}
                     \right)
                      \right) \\
                         = &
                       \frac{1}{2}
                       \phi \big(  
                       c_1( c_1 + 2 c_2 ) + (2c_1 + c_2 )( c_2^\ast + 2 s_1)
                       - c_2^\ast (c_1 + c_2 )  - (2c_2^\ast + s_1 ) ( c_2^\ast + s_1)
                       \big)
                        = \frac{1}{4}.
                     \end{align*}                                
    \end{proof}
    
    \begin{remark}\label{rem:4.2}
   There exists families of permutations
      $ (\sigma_N)_N $ 
      and 
      $ ( \tau_N )_N $, each
       $ \sigma_N, \tau_N $
        from 
       $ \mathcal{S}([ N ]^2) $,
      satisfying condition \emph{(i)} but not condition \emph{(ii)} 
      of Theorem \ref{thm:1}, such that
      $ G_N^{\sigma_N} $ 
      and
      $ G_N^{\tau_N } $
      are asymptotically free.
    \end{remark}
    
    \begin{proof}
    We shall show the property for each 
    $ \sigma_N $
     the identity permutation and
    $ \tau_N $
    given by
    \begin{align*}
    \tau_N( i, j) = ( \varphi_N ( i + j), j ) 
    \end{align*}
    where, as in the proof of Remark \ref{remark:4:1}, 
    $ \varphi_N : [ 2 N ] \rightarrow [ N ] $
    is given by
   $  \varphi( i) = k $
    if and only if 
     $ k \equiv i $ mod  $ N $.
 We have that 
        \begin{align*}
         \mathfrak{j} ( \tau_N : \tau_N )
         & =
          \#\big\{ (i, j, k):\
          ( \varphi_N ( i + j), j ) = ( k, \varphi_N ( j +  k ) ) 
          \big\}\\
           & = 
           \#\big\{ (i, j, k):\  \varphi_N ( j +  k ) = j 
           \textrm{ and }
           \varphi_N ( i + j) = k 
           \big\}\\
           & = 
                  \#\big\{ (i, j, k):\  k = N = \varphi_N ( i + j ) 
                  \big\}
                   = O(N)
        \end{align*}
        and
    \begin{align*}
    \#\big\{ (i, j, k): \tau_N (i, j) \in  \{ (i, k),  ( k, j) \} 
    \big\} & \geq
    \#\big\{ (i, j, k): ( \varphi_N ( i + j), j )  = (k, j) \big\}\\
    \geq &
    \big\{
    (i, j, k): k = \varphi_N ( i + j )
    \big\}
     = N^2
    \end{align*}
  hence
     $ \sigma_N $ 
    and
     $ \tau_N $ 
     satisfy condition (i) but
     do not satisfy condition (ii) of Theorem \ref{thm:1}.
     
     Let
      $ m $ 
      be a positive integer and let
      $ \overrightarrow{\sigma_N }= ( \sigma_{1, N}, \sigma_{2, N}, \dots, \sigma_{ m, N})  $
      be such that
   $ ( \sigma_{k, N})_N \in
    \big\{  
    (\textrm{Id}_N )_N , ( \tau_N)_N , ( t\circ \tau_N \circ t )_N  
   \big\} $.  
As in the proof of Theorem \ref{thm:1}, using that
 $ ( G_N^{ \tau_N} )^\ast = G_N^{ t \circ \tau_N \circ t } $,
 to show the asymptotic free independence of 
$ G_N $ 
and
$ G_N^{\tau_N } $
it suffices to prove that   
\begin{align} \label{V:tau}
\lim_{N \rightarrow \infty } \mathcal{V} ( \pi, \overrightarrow{\sigma_N } ) = 0
\end{align}  
whenever
 $ \pi \in NC_2( m ) $ 
 has the property that 
 $ \sigma_{k, N} = \textrm{Id}_N $
 and 
 $ \sigma_{ l , N } \neq \textrm{Id}_N $
 for some
  $ (k, l) \in \pi $.
  
  We shall prove (\ref{V:tau}) by induction on 
   $m $.
  If 
   $ m = 2 $, 
   then 
   $ \pi = (1, 2) $
   and
   \begin{align*}
   \mathcal{V} ( \pi, \overrightarrow{\sigma_N}) = &
   \frac{1}{N} \sum_{ {i, j} \in [ N ] }
   E \big( g_{\sigma_{1, N}(i, j) } g_{\sigma_{2, N} (j, i) } \big)\\
   \leq &\frac{1}{N^2} \# \big\{
   (i, j) \in [ N ] :\ (i, j) \in \{ ( \varphi_N (i + j), j ), ( i, \varphi_N ( i + j ) )
   \}
   \big\}
    = \frac{2N}{N^2}.
    \end{align*}
    
    For the induction step, note that 
    $ \pi $ 
    has a block consisting of two consecutive elements 
    $ ( k, k + 1) $. 
    If 
    $ \sigma_{ k, N } = t \circ \sigma_{k + 1, N }\circ t $,
    then the result follows from Lemma \ref{lemma:2:5}.
    If
     $ \sigma_{k, N} = \sigma_{k+1 N}
      \in \{ \tau_N, t \circ \tau_N \circ t \} $,
      then 
      $ \mathfrak{j} ( \tau_N: \tau_N )
       = \mathfrak{j} ( t \circ \tau_N \circ t : t \circ \tau_N \circ t ) = o(N^2) $ 
      and Lemma \ref{lemma:ik1} give that
       $ \displaystyle \lim_{ N \rightarrow \infty} 
       \mathcal{V}( \pi, \overrightarrow{ \sigma_N }) = 0 $.
      So it suffices to show that (\ref{V:tau}) holds true when
      one of the permutations 
      $ \sigma_{ k, N} $,
      $ \sigma_{ k + 1, N } $
      is 
      $ \textrm{Id}_N $
      and the other is 
      $ \tau_N $
       or
        $ t \circ \tau_N \circ t $.
       Eventually taking adjoints and a circular permutation, we can suppose that
        $ k = m-1 $,
        i.e. 
        $ \pi = \pi^\prime \oplus (m -1, m ) $,
         and that
        $ \sigma_{k, N} = \textrm{Id}_N $.
        
        Similarly to the proof of Lemma \ref{lemma:ik}, consider the sets 
          $ \{ P_k \}_{ 1 \leq k \leq \frac{m}{2} } $
        given by
           $ P_1 = \{ a(1), \pi ( a (1))\} $
           and 
           $ P_{k +1} =  P_k \cup \{ a(k+1), \pi ( a (k+1))\} $
           where the sequence
            $ \{ a(k)\} _{1 \leq k \leq \frac{m}{2} } $
            is given by
            $ a(1) = 1 $ 
            and 
            $ a(k + 1) = \min \{ l \in [ m ] \setminus P_k \} $.
 The condition 
        $ \pi = \pi^\prime \oplus (m-1, m) $
        gives that
       $ P_{ \frac{m}{2} - 1} = [ m -2 ]  $   
       and
       $ a ( \frac{m}{2}) = m-1 $. 
       
       Again, as in the proof of Lemma  \ref{lemma:ik}, we define 
define
    \begin{align*}
    A_k = \{
     \overrightarrow{\alpha} \in [ N]^{4k} :
    \overrightarrow{\alpha} = \oi [ P_k ] 
    \textrm{ for some }
    \oi \in I(N, m)
    \textrm{ such that }
    v ( \pi, \overrightarrow{\sigma}, \oi) \neq 0
    \},
    \end{align*}
  and  properties 
  $(\mathfrak{p}.1)$, 
  $(\mathfrak{p}.2)$
   and
    $(\mathfrak{p}.4)$
     give that  
     \begin{align*}
     \# A_{ \frac{m}{2} -1} \leq N^{\frac{m}{2}}.
     \end{align*} 
     It suffices to show that for each 
  $ \overrightarrow{\alpha} \in  A_{ \frac{m}{2} -1} $
  there exists a unique 
  $ \overrightarrow{i} \in I(N, m) $
  such that 
   $   v ( \pi, \overrightarrow{\sigma}, \oi) \neq 0 $
   and
  $ \overrightarrow{\alpha} = \oi [ P_k ] $.
  This implies that 
  $ \# A_{ \frac{m}{2} } = \# A_{ \frac{m}{2} -1 } $, 
 and (\ref{V:tau}) follows, since
  \begin{align*}
  \mathcal{V} ( \pi, \overrightarrow{ \sigma_N }) 
  =
  N^{ -\frac{m}{2} -1} \cdot \# A_{ \frac{m}{2} } = o(N).
  \end{align*} 
  
  Fix
   $ \overrightarrow{\alpha} \in  A_{ \frac{m}{2} -1} $ 
   and let
   $ \overrightarrow{i} = 
   ( \overrightarrow{\alpha},
    i_{ m-1}, i_{ - (m-1) }, i_m, i_{ -m} ) \in I(N, m) $
    such that 
     $   v ( \pi, \overrightarrow{\sigma}, \oi) \neq 0 $
     and
    $ \overrightarrow{\alpha} = \oi [ P_k ] $.
     The condition 
           $   v ( \pi, \overrightarrow{\sigma}, \oi) \neq 0 $
           gives that
            \begin{align*}
            \sigma_{1, N}( i_{ m-1}, i_{ -(m-1)}) 
            = 
            t \circ \sigma_{2, N} ( i_{ m }, i_{ - m }). 
            \end{align*}
    Note that
    $ i_{ -(m-1)} = i_{m} $
    from the definition of 
    $ I( N, m) $,
     and that
     $ i_1 = i_{ -m }$ 
     and
      $ i_{-(m-2)} = i_{ m -1} $
      are components of 
      $ \overrightarrow{\alpha} $, 
      si they are fixed. 
    Since 
    $ \sigma_{1, N} = \textrm{Id}_N $,
    we obtain     
    \begin{align}
    \label{last}
    (i_{ m -1}, i_m ) = t \circ \sigma_{2, N} ( i_m, i_{ -m }) 
    \end{align}
 If 
 $ \sigma_{2, N} = \tau_N $,
  equation (\ref{last}) gives that
  $ i_{ m } = \varphi_N ( i_m + i_{ - m }) $;
  if
  $ \sigma_{2, N} = t \circ \tau_N \circ t $, 
    equation (\ref{last}) gives that
    $ i_{ m -1 } = \varphi_N ( i_m + i_{ - m } ) $.
    In either case, 
    $ i_m $ 
    is uniquely determined by the pair
    $ ( i_{ m-1}, i_{ - m }) $,
    i.e. by
    $ \overrightarrow{ \alpha} $, 
    so the proof is complete.
    \end{proof}

         %



\bibliographystyle{alpha}

\begin{thebibliography}{10}
	
	\bibitem{aubrun} G. Aubrun, S. Szarek,  E. Werner, 
	\emph{Hastings's additivity counterexample via Dvoretzky's theorem},
	 Commun. Math. Phys. 305(1), 85–97 (2011)
	
	\bibitem{banica-nechita} T. Banica and I. Nechita, 
	\emph{Asymptotic eigenvalue distributions of block-transposed Wishart matrices}, J. Theor. Probab. February 2012, 1–15 (2012)
	
	\bibitem{billiard} T. Damour, S. De Buyl, M. Henneaux, C. Schomblond,
\emph{Einstein billiards and overextensions of finite-dimensional simple Lie algebras}, J. of High Energy Phys., Vol. 2002, JHEP08(2002)	

\bibitem{effros-popa} E.G. Effros, M. Popa,
\emph{Feynman diagrams and Wick products associated with q-Fock spaces}, Proc. Natl. Acad. Sci. U.S.A.  100(15), 2003;  8629--8633.


\bibitem{horod} M. Horodecki, P. Horodecki, and R. Horodecki, \emph{Separability of mixed states: Necessary and sufficient conditions}, Phys. Lett. A 223(1–2), 1–8 (1996). 

\bibitem{jason}
S. Janson, \emph{Gaussian Hilbert Spaces}, Cambridge Tracts in Mathematics, vol. 129, Cambridge University Press, Cambridge, 1997.

\bibitem{zycz} 
A. Mandarino, T. Linowski, K. \.{Z}yczowski,
\emph{Bipartite unitary gates and billiard dynamics in the Weyl chamber},
Phys. Rev. A 98, 012335 (2018), arXiv:1710.10983 [quant-ph]

\bibitem{mingo-popa-transpose} J.A. Mingo, M. Popa, \emph{Freeness and the transposes of unitarily invariant random matrices}, Journal of Funct. Anal. 271(4), 2014, 883--921

\bibitem{mingo-popa-wishart}
J.A. Mingo, M. Popa, \emph{Freeness and the partial transpose of Wishart random matrices}, Canad. J. Math., Vol. 71(3), 2019; 659--681.

\bibitem{mingo-popa-wishart2} 
J.A. Mingo, M. Popa, \emph{Freeness and the partial transpose of Wishart random matrices. Part II}, preprint   

\bibitem{nica-speicher}
A. Nica, R. Speicher, \emph{Lectures on the Combinatorics of Free Probability},
London Mathematical Society Lecture Note Series, vol. 335,
Cambridge University Press, 2006

\bibitem{popa-hao} M. Popa, Zh. Hao, \emph{A Combinatorial Result on Asymptotic Independence Relations for Random Matrices with Non-Commutative Entries}, 
J. Operator Theory 80 (2018), no. 1, 47--76

\bibitem{popa-hao-boolean} \emph{An asymptotic property of large matrices with identically distributed Boolean independent entries},
Infin. Dimens. Anal., Quantum Probab. Relat. Top., Vol. 22, No. 04, 1950024 (2019)

\end{thebibliography}


\end{document}